\documentclass[11pt]{amsart}
\usepackage{amsmath}
\usepackage{amscd}
\usepackage{url}
\usepackage{listings}
\usepackage{amsrefs}
\usepackage{amssymb}
\usepackage{amsxtra}

\usepackage{mathabx}
\usepackage{mathrsfs}
\usepackage{stmaryrd}
\usepackage[utf8]{inputenc}
\usepackage{eucal}
\usepackage{fullpage}
\usepackage{hyperref}
\hypersetup{pdfauthor=Valentin Blomer and Etienne Fouvry and Emmanuel
  Kowalski and Philippe Michel and Djordje Milicevic, pdftitle=Remarks
on bilinear forms with Kloosterman sums}

\renewcommand{\leq}{\leqslant}
\renewcommand{\geq}{\geqslant}
 
\numberwithin{equation}{section}

\newcommand{\uple}[1]{\text{\boldmath${#1}$}}

\def\stacksum#1#2{{\stackrel{{\scriptstyle #1}}
{{\scriptstyle #2}}}}

\newcommand{\bexp}{1/20}

\newcommand{\Cc}{\mathbf{C}}

\newcommand{\Zz}{\mathbf{Z}}

\newcommand{\Rr}{\mathbf{R}}

\newcommand{\Fq}{{\mathbf{F}_q}}
\newcommand{\Fqt}{{\mathbf{F}^\times_q}}

\newcommand{\mcN}{\mathcal{N}}

\newcommand{\mods}[1]{\,(\mathrm{mod}\,{#1})}

\newcommand{\fourier}[1]{\widehat{{#1}}}

\newcommand{\what}{\widehat}

\newcommand{\bessel}[1]{\widecheck{{#1}}}

\DeclareMathOperator{\Kl}{\mathrm{Kl}_2}

\DeclareMathOperator{\ET}{\mathrm{ET}}

\DeclareMathOperator{\supp}{supp}

\newcommand{\eps}{\varepsilon}
\renewcommand{\rho}{\varrho}

\DeclareMathSymbol{\gena}{\mathord}{letters}{"3C}
\DeclareMathSymbol{\genb}{\mathord}{letters}{"3E}

\newcounter{bnd}

\theoremstyle{plain}
\newtheorem{theorem}{Theorem}[section]
\newtheorem*{theorem*}{Theorem}
\newtheorem{lemma}[theorem]{Lemma}

\newtheorem{proposition}[theorem]{Proposition}

\theoremstyle{remark}

\theoremstyle{definition}

\newtheorem{remark}[theorem]{Remark}

\newcommand{\mcM}{\mathcal{M}}

\newcommand{\lf}{\lambda_f}

\renewcommand{\geq}{\geqslant}
\renewcommand{\leq}{\leqslant}

\newcommand{\ov}[1]{\overline{#1}}

\newcommand\sumsum{\mathop{\sum\sum}\limits}

\begin{document}
 
\title{Some applications of smooth bilinear forms with Kloosterman
  sums}

\author{Valentin Blomer}
\address{Mathematisches Institut, Universit\"at G\"ottingen,
  Bunsenstr. 3-5, 37073 G\"ottingen, Germany} \email{vblomer@math.uni-goettingen.de}

\author{\'Etienne Fouvry}
\address{Laboratoire de Math\'ematiques d'Orsay, Universit\' e Paris--Saclay  \\
    91405 Orsay  \\France}
\email{etienne.fouvry@math.u-psud.fr}

\author{Emmanuel Kowalski}
\address{ETH Z\"urich -- D-MATH\\
  R\"amistrasse 101\\
  CH-8092 Z\"urich\\
  Switzerland} \email{kowalski@math.ethz.ch}

\author{Philippe Michel} \address{EPF Lausanne, Chaire TAN, Station 8, CH-1015
  Lausanne, Switzerland } \email{philippe.michel@epfl.ch}

 \author{Djordje Mili\'cevi\'c}
 \address{Department of Mathematics,
Bryn Mawr College,
101 North Merion Avenue,
Bryn Mawr, PA 19010-2899, U.S.A.}
 \curraddr{Max-Planck-Institut f\"ur Mathematik, Vivatsgasse 7, D-53111 Bonn, Germany}
 \email{dmilicevic@brynmawr.edu}

 \thanks{V.\ B.\ was partially supported by the  
   Volkswagen Foundation.  \'E.\ F.\ thanks ETH Z\"urich and EPF Lausanne
   for financial
   support. Ph. M. was partially supported by the SNF (grant
   200021-137488) and the ERC (Advanced Research Grant 228304). V.\ B.,
   Ph.\ M.\ and E.\ K.\ were also partially supported by a DFG-SNF lead
   agency program grant (grant 200021L\_153647). D. M. was partially
   supported by the NSF (Grant
   DMS-1503629) and ARC (through Grant DP130100674).}



\begin{abstract}

  We revisit a recent bound of I. Shparlinski and T. P. Zhang on
  bilinear forms with Kloosterman sums, and prove an extension for
  correlation sums of Kloosterman sums against Fourier coefficients of
  modular forms. We use these bounds to improve on earlier results on
  sums of Kloosterman sums along the primes and on the error term of
  the fourth moment of Dirichlet $L$-functions.
\end{abstract}

\maketitle

\setcounter{tocdepth}{1}


\section{Statement of results}\label{intro}

\subsection{Preliminaries}

This note is motivated by a recent result of I. E. Shparlinski and
T. P. Zhang \cite{SZ} concerning bilinear forms with Kloosterman
sums. Given a prime $q$ and $m\in\Fq$, let
\[
\Kl(m;q):=\frac{1}{\sqrt{q}}\sum_{\substack{x\in\Fqt\\ xy=1}}e_q(y+mx)
\]
denote the normalized Kloosterman sum, where
$e_q (x) =\exp( 2 \pi i x /q)$. Shparlinski and Zhang (\cite [Theorem
3.1]{SZ}) proved the following theorem.

\begin{theorem}[Shparlinski--Zhang]\label{thmSZ} 
  Let $q$ be a prime number and let $\mcM,\mcN\subset [1,q-1]$ be
  intervals of lengths $M,N\geq 1$. Then we have
\begin{equation}\label{SZbound1}
  \sumsum_{m\in{\mcM},n\in\mcN}\Kl(mn;q)\ll_{\eps} q^{\eps}\Bigl( q^{1/2}+\frac{MN}{q^{1/2}} \Bigr)
\end{equation}
for any $\eps > 0$, where the implied constant depends only on $\eps$.
\end{theorem}

In light of the Weil bound for Kloosterman sums $|\Kl(m;q)|\leq 2$, the
estimate \eqref{SZbound1} is non-trivial as long as $MN$ is a bit larger than
$q^{1/2}$. On the other hand, if $M$ or $N$ is close to $q$, other
methods (e.g.\ the completion method) become  more efficient. In particular,
the restriction that $M$ and $N$ are $\leq q$ is not really
restrictive for applications.

The aim of this paper is two-fold. On the one hand, we put  Theorem \ref{thmSZ}
into a slightly more general context in Propositions~\ref{propSZ} and \ref{propIScuspidal}; viewing it as a correlation estimate for Kloosterman sums and a divisor function (which itself is a Fourier coefficient of an Eisenstein series), it turns out to be a consequence of a version of the Voronoi summation formula. On the other hand, we give two applications of independent interest to the fourth moment of Dirichlet $L$-functions in Theorem~\ref{472} and sums of Kloosterman sums over primes in Theorem~\ref{Klpq}; these applications are discussed in Subsection \ref{appl}.

\subsection{Variations on a theme} Our first result is a smoothed version of the bound \eqref{SZbound1}.
To state it, we use the following class of smoothing functions. For a
modulus $q\geq 1$ and a parameter $Q\geq 1$, we will consider
functions satisfying the following conditions:
\begin{equation}\label{Wbound}
\begin{split}
	W\colon [0,+\infty[\to\Cc\text{ is smooth, }
  \mathrm{Supp}(W)\subset [1/2,2],\\
  W^{(j)}(x)\ll_{j, \eps} \bigl(q^{\eps}Q\bigr)^j\,
  \text{ for any $x\geq 0$, $j\geq 0$ and $\eps>0$.}
\end{split}
\end{equation}

\begin{proposition}\label{propSZ} Let $q$ be a prime number and let
  $Q\geq 1$ be a real number. Let $W_1,W_2$ be functions
  satisfying~\eqref{Wbound}.  For any $M,N\geq 1$ and any integer $a$
  coprime with $q$, we have
\begin{equation}\label{SZboundsmooth}
\sumsum_{m,n}W_1\Bigl(\frac{m}M\Bigr)W_2\Bigl(\frac{n}N\Bigr)\Kl(amn;q)\ll_\eps  
(qQ)^{\eps}\, Q^2\,\Bigl(q^{1/2}+\frac{MN}{q^{1/2}}\Bigr).	
\end{equation}
Furthermore, if $W_3$ also satisfies~\eqref{Wbound}, then for any
$Y\geq 1$, we have
\begin{equation}\label{SZboundsmooth1}
  \sumsum_{m,n}W_1\Bigl(\frac{m}M\Bigr)
  W_2\Bigl(\frac{n}N\Bigr)W_3\Bigl( \frac{mn}{Y}\Bigr)\Kl(amn;q)\ll_\eps  
  (qQ)^{\eps}\, Q^2\,\Bigl(q^{1/2}+\frac{MN}{q^{1/2}}\Bigr).
\end{equation}
In both cases, the implied constant depends only on $\eps$.
\end{proposition}
The inequalities \eqref{SZboundsmooth} and \eqref{SZboundsmooth1}
could be easily deduced from the result of Shparlinski and Zhang by
summation by parts with respect to the variables $m$ and $n$. In \S
\ref{par2}, we will give an alternative proof based on~\cite[Prop.~2.2]{FKMd3}.
The $Q$-dependence in Proposition~\ref{propSZ} is presented in a compact form well suited for our applications but it is not fully optimized otherwise (in particular, for Theorem~\ref{472} we will be using $Q=q^{\varepsilon}$); our proof actually yields a better $Q$-dependence in some other ranges.

We can view the bounds \eqref{SZboundsmooth} and
\eqref{SZboundsmooth1} essentially as sums over a single variable weighted by the
divisor function $d$. The advantage of our proof of Proposition \ref{propSZ} is that it provides naturally an automorphic generalization, where the divisor function is replaced with Fourier 
coefficients of modular forms.

\begin{proposition}\label{propIScuspidal} Let $(\lf(n))_{n\geq 1}$ be
  the Hecke eigenvalues of a holomorphic cuspidal Hecke eigenform $f$
  of level $1$, normalized so that $|\lf(n)|\leq d(n)$.  Let $q$ be a
  prime number, and let $W$ be a function satisfying~\eqref{Wbound}
  with $Q=1$. Let $a$ be an integer coprime to $q$. For any $N\geq 1$
  and any $\eps>0$, we have
\begin{equation}\label{SZcusp}
  \sum_{n\geq 1}\lf(n)\Kl(an;q)
  W\Bigl(\frac{n}N\Bigr)\ll_{\eps, f} (qN)^{\eps}\Bigl(q^{1/2}+\frac{N}{q^{1/2}}\Bigr) 	
\end{equation}
where the implied constant depends only on $f$ and $\eps$.
\end{proposition}

\begin{remark}
  This is by no means the most general statement that may be proved
  along these lines.
\end{remark}

As pointed out in \cite{SZ}, the estimates \eqref{SZbound1} and
\eqref{SZboundsmooth} are significant improvements of the bound
\begin{equation}\label{FKMbound1} 
  \sumsum_{m,n}W_1\Bigl(\frac{m}M\Bigr)W_2\Bigl(\frac{n}N\Bigr)\Kl(amn;q)\ll_{\eps, Q}  
  q^{\eps}MN\Bigl(1+\frac{q}{MN}\Bigr)^{1/2}q^{-1/8},	
\end{equation}
and likewise the estimate \eqref{SZcusp} improves significantly over
\begin{equation}\label{FKMbound2} \sum_{n\geq 1}\lf(n)\Kl(an;q)W\Bigl(\frac{n}N\Bigr)\ll_{\eps,Q}  
  q^{\eps}N\Bigl(1+\frac{q}{N}\Bigr)^{1/2}q^{-1/8},
\end{equation}
both of which were obtained by Fouvry, Kowalski and Michel as special
cases of \cite[Thm.~1.16]{FKM2} and \cite[Thm.~1.2]{FKM1}. 

\subsection{Applications}\label{appl} The bounds \eqref{FKMbound1} and
\eqref{FKMbound2} 
have been applied recently in a number of problems, and the bounds
\eqref{SZbound1} and \eqref{SZboundsmooth} lead to further
improvements. The main source for these improvements is the new input of Proposition \ref{propSZ}, but a bit of extra work is necessary. 

As a first application, we can improve our work on the error term for
the fourth moment of Dirichlet series $L(s, \chi)$ of characters
$\chi$ to a prime modulus $q$ (\cite[Theorem 1.1]{445}).
\begin{theorem}\label{472}
  There exists a polynomial $P_4\in \Rr[X]$ of degree $4$, such that
\[
\frac{1}{q-1}\sum_{\chi\mods q}|L(\chi,1/2)|^4=P_4(\log
q)+O(q^{-\bexp+\eps})
\]
for all primes $q$, where the implied constant depends only on
$\eps>0$.  If the Ramanujan--Petersson conjecture holds for Fourier
coefficients of Hecke--Maa{\ss} forms of level $1$, then the exponent $\bexp$
can be replaced by $1/16$.
\end{theorem}
\begin{remark} In \cite[Theorem 1.1]{445}, the exponents were
  respectively $1/32$ (unconditionally) and $1/24$ (assuming the
  Ramanujan--Petersson conjecture). The first breakthrough in this
  respect is due to M.\ Young \cite{MY} who obtained an asymptotic
  formula with exponents $5/512$ (resp.\ $1/80$).
\end{remark}

\begin{remark}
  The proof of Theorem  \ref{472} follows the same lines as \cite[\S
  6.3]{445}, except that instead of the bound \cite [(5.5)]{445}
  (i.e.\ \eqref{FKMbound1} above) we use Proposition \ref{propSZ}. It
  is of some interest to record here in outline how this improved
  exponent arises. The problem of the fourth moment leads to
  evaluating non-trivially the shifted convolution type sum
\begin{equation}\label{eqquadrismooth}
\sumsum_\stacksum{m\asymp M,n\asymp N}{m\equiv n\mods q}d(m)d(n)\approx
\sumsum_\stacksum{m_1,m_2,n_1,n_2}{m_1m_2\equiv n_1n_2\mods q}1	
\end{equation}
with $d$ the usual divisor function and with
$MN=M_1M_2N_1N_2\approx q^2$. The spectral theory of automorphic forms
provides a good error term when $M$ and $N$ are relatively
close in the logarithmic scale. Otherwise, assuming that $N=N_1N_2\geq M=M_1M_2$, we apply the
Poisson summation formula to both variables $n_1$ and $n_2$
(equivalently, the Voronoi summation formula applied to the variable
$n=n_1n_2$), getting two variables of dual size $n_1^*\sim q/N_1$ and
$n_2^*\sim q/N_2$ and a smooth quadrilinear sum of Kloosterman sums
\[ \sumsum_{m_1,m_2,n^*_1,n^*_2}\Kl(m_1m_2n_1^*n_2^*;q), \]
which is evaluated by various means, in particular using the smooth
bilinear sum bound \eqref{SZboundsmooth}.  In our specific case, the
bound \eqref{SZboundsmooth} amounts to applying the Poisson formula to
two of the four variables $m_1,m_2,n^*_1,n^*_2$. This leads back to a
sum of the type \eqref{eqquadrismooth}, which is then bounded
trivially. This argument is not circular, and allows for an
improvement, because we (implicitly) apply the process to variables
different from the ones we started from (for instance to $m_1$ and
$n_1^*$ instead of $n_1^*$ and $n_2^*$).
\end{remark}

Our second application is an improvement of the first bound in
\cite[Cor.~1.13] {FKM2} for Klooster\-man sums over primes in short
intervals:

\begin{theorem}\label{Klpq} Let $q$ be a prime number. 
  Let $Q\geq 1$ be a parameter and let $W$ be a function
  satisfying~\eqref{Wbound}. Then for every $X$ such that
  $2\leq X\leq q$ and every $\varepsilon >0$, we have
  \begin{equation}\label{evening2}
    \sum_{p\text{ prime }} 
    W \Bigl( \frac{p}{X}\Bigr) \Kl (p;q) \ll_{\varepsilon} 
    q^{1/4+\varepsilon} Q^{1/2} X^{2/3}.
 \end{equation}
 \par
 In addition, for every prime $q$, every $X$ such that $2\leq X\leq q$
 and every $\varepsilon >0$, we have
 \begin{equation}\label{evening1}
 \sum_\stacksum{p\leq X}{p\text{ prime}}\Kl(p;q)\ll_{\eps} q^{1/6+\eps}\, X^{7/9}.
 \end{equation}
 In both cases, the implicit constant depends only on $\varepsilon$.
\end{theorem}

 \begin{remark} The range where these bounds are non-trivial is the
   same as that in~\cite[Cor.~1.13]{FKM2}, namely the length of
   summation $X$ should be greater than $q^{3/4 +\varepsilon}$ if $Q$
   is fixed. The improvement therefore lies in the greater
   cancellation in this allowed range. For instance, when $X=q$, we
   gain a factor $q^{1/18 -\varepsilon}$ over the trivial bound for
   the sum appearing in \eqref{evening1} instead of
   $q^{1/48 -\varepsilon}$ in \cite [Corollary 1.13]{FKM2}.
  \end{remark}

\subsection*{Acknowledgement.} We would like to thank the referee for very useful suggestions that improved the presentation of the paper. 

\section{Correlation sums of Kloosterman sums and divisor-like
  functions}\label{par2}

In this section, we revisit Theorem \ref{thmSZ} and establish Proposition \ref{propSZ}.
The idea behind the proof of Theorem \ref{thmSZ} is that after applying the completion method twice over the $m$ and $n$ variables, the Kloosterman sum $\Kl(amn;q)$ is transformed into the Dirac type function
$q^{1/2}\delta_{mn\equiv a\mods q}$, and taking the congruence condition into account one saves (in the most favourable situation) a factor $q^{1/2}/q=q^{-1/2}$ over the trivial bound.

In our smoothed setting, the completion method is replaced by two
applications of the Poisson summation formula or more precisely by a
single application of the \emph{tempered Voronoi summation formula} of
Deshouillers and Iwaniec, in the form established in
\cite[Prop.~2.2]{FKMd3}.

Let $q$ be a prime number, and let $K\colon\Zz\to \Cc$ be a
$q$-periodic function.  The \emph{normalized Fourier transform} of $K$
is the $q$-periodic function on $\Zz$ defined by
\begin{equation*} 
  \fourier{K}(h) = \frac{1}{\sqrt{q}}\sum_{n\bmod q} K(n) e_q(
  {hn})
\end{equation*}
and the \emph{Voronoi transform} of $K$ is the $q$-periodic function
on $\Zz$ defined by
\[
\bessel{K}(n) = \frac{1}{\sqrt{q}}\sum_{\substack{h\bmod q\\(h,q) =1}}
\fourier{K}(h) e_q ({\overline h n} ).
\]

\begin{proposition}[Tempered Voronoi formula modulo
  primes]\label{Voronoigeneral0} 
  Let $q$ be a prime number, let $K\, :\ \Zz \longrightarrow \Cc$ be a
  $q$-periodic function, and let $G$ be a smooth function on $\Rr^2$
  with compact support and Fourier transform denoted by $\what G$. We have
\begin{equation}\label{voronoi} 
\sumsum_{m,n\in\Zz}  K(mn) G(m,n)=\frac{\fourier{K}(0)}{\sqrt{q}}
\ \sumsum_{m,n\in\Zz}  G(m,n)+
\frac{1}{q}\ \sumsum_{m,n\in\Zz} \bessel{K}(mn)\fourier{G}\Bigl(\frac{m}q,\frac
nq\Bigr).
\end{equation}
\end{proposition}

The key point is that when $K$ is a (multiplicatively shifted)
Kloosterman sum, then $\widecheck{G}$ is a normalized delta-function:

\begin{lemma}\label{immediate} For $(a,q)=1$ and $K(n)=\Kl(an;q)$ one has
\[
\what K(h)=\begin{cases}0&\text{if $q\mid h$},\\
e_q(-a\ov h)	&\text{if $q\nmid h$,}
\end{cases}
\]
and
\[
\bessel{K}(n)=
\begin{cases}\displaystyle{\frac{q-1}{q^{1/2}}}&\text{if $n\equiv a \bmod q$},\\
  \displaystyle{-\frac{1}{q^{1/2}}}&\text{otherwise}.
\end{cases}
\]
\end{lemma}

This lemma is proved by an immediate computation.  We now begin with
the proof of \eqref{SZboundsmooth}. Let $q$ be a prime and let $W$ be
a function satisfying~(\ref{Wbound}). By integration by parts, we then
have
\[
\what W (t) \ll_{j,\varepsilon} \min \bigl( 1, q^{j\varepsilon}\vert
t/Q\vert^{-j}\bigr)
\]
for $t\in \Rr$ and for any integer $j\geq 0$ and $\varepsilon >0$,
where the implied constant depends only on $j$ and $\eps$.

Defining $G(m,n)=W_1(m/M) W_2(n/N)$, we deduce that for any $A$ and
any $\eps>0$, we have
\begin{equation}\label{G}
  \what G\Bigl(\frac{m}q,\frac{n}q\Bigr)=M\what W_1\Bigl(\frac{mM}q\Bigr)N\what W_2\Bigl(\frac{nN}q\Bigr)\ll_{\eps,A} q^\eps MN\Bigl(1+\frac{\vert m\vert M}{qQ}\Bigr)^{-A}\Bigl(1+\frac{\vert n\vert N}{qQ}\Bigr)^{-A}.
\end{equation}
We next apply the Voronoi formula, Proposition~\ref{Voronoigeneral0}, with $K(n)=\Kl(an;q)$ to the left-hand side of
\eqref{SZboundsmooth}.  The first term on the right-hand side of
\eqref{voronoi} vanishes since $\fourier{K}(0) = 0$. By Lemma \ref{immediate}
and \eqref{G}, the contribution of $mn \not\equiv a$ (mod $q$) in the
second term is  of order at most
\begin{align*}
  \frac{MN}{q^{3/2-\varepsilon}} \sum_{m, n\in\Zz} \Bigl(1+\frac{\vert m\vert M}{qQ}\Bigr)^{-2}\Bigl(1+\frac{\vert n\vert N}{qQ}\Bigr)^{-2}& \ll \frac{MN}{q^{3/2-\varepsilon}}\Bigl(1 + \frac{qQ}{M}\Bigr) \Bigl(1 + \frac{qQ}{N}\Bigr) \\
&\ll q^{\varepsilon}\Bigl(\frac{MN}{q^{3/2}} + \frac{(M+N)Q}{q^{1/2}} + q^{1/2}Q^2\Bigr).
\end{align*}

Similarly, the remaining terms $mn \equiv a$ (mod $q$) are, up to a constant,  bounded by 
\[  q^\eps\frac{MN}{q^{1/2}} \sum_{n \equiv a\, (\text{mod }q)} d(n) \Bigl(1 + \frac{n MN}{q^2Q^2}\Bigr)^{-2} \ll  (q^2Q)^{\varepsilon} \left(\frac{MN}{q^{1/2}} +  Q^2 q^{1/2}\right). \]
This completes the proof of  \eqref{SZboundsmooth}.

Next, we prove \eqref{SZboundsmooth1}. We may suppose that
\begin{equation*}
MN/8 <Y <8 MN,
\end{equation*}
since otherwise the sum of interest is empty. Then we see that for
$M/2 < x<2M$ and $N/2<y<2N$, we have the inequalities

\[
\frac{\partial^{i+j} W_3(xy/Y)}{\partial x^i \, \partial y^j}  \ll_{\varepsilon, i, j} (q^\varepsilon Q)^{i+j} M^{-i}N^{-j}
\]
for all non-negative integers $i, j$. Hence the function
$ G(x,y) = W_1(x/M) W_2(y/N)
W_3(xy/Y)$ 
satisfies the inequalities
\[
\frac{\partial^{i+j} G(x,y)}{\partial x^i \partial y^j}\ll_{\varepsilon,i,j} (q^\varepsilon Q)^{i+j} x^{-i} y^{-j},
\]
for $x$, $y >0$, $\varepsilon >0$ and integers $i, j \geq
0$.  By repeated integration by parts of the definition of the Fourier
transform
\[
\what G (u,v)= \int_{-\infty}^{\infty} \int_{-\infty}^\infty G(x,y) e( -ux-vy) dx\, dy,
\]
we obtain the bound
\[
\what G\Bigl( \frac{m}{q}, \frac{n}{q}\Bigr) \ll_{\varepsilon, A} q^\varepsilon MN \Bigl( 1 +\frac{\vert m\vert M}{qQ}\Bigr)^{-A}
 \Bigl( 1 +\frac{\vert n\vert N}{qQ}\Bigr)^{-A} 
\]
for any $A$ and any $\varepsilon >
0$, analogously to \eqref{G}.  The end of the proof of
\eqref{SZboundsmooth1} is now similar to \eqref{SZboundsmooth}.

For future reference we record the following bound for type II sums of
Kloosterman sums 
\cite[Thm.~1.17]{FKM2}.

\begin{proposition}\label{proptypeII} Let $q$ be a prime number. 
  Let $1\leq M,N\leq q$ and $(\alpha_m)$, $(\beta_n)$ be sequences of
  complex numbers supported in $[M,2M]$ and $[N,2N]$ respectively. Let
  either $Q=1$ and $W$ be the constant function $1$, or $Q\geq 1$ and $W$
 be a function satisfying~\eqref{Wbound}. Then, for every
  $\varepsilon >0$, we have
\[
\sumsum_{m,n}\alpha_m \beta_n \Kl (mn;q)W
\Bigl(\frac{mn}{Y}\Bigr)\ll_\varepsilon \Vert \mathbf \alpha\Vert_2 \,
\Vert \mathbf \beta \Vert_2\, (MN)^{1/2}\Bigl( \frac1M+ Q\frac{q^{1/2+
    \varepsilon}}{N} \Bigr)^{1/2}.
\]
\end{proposition}

This is a special case of \cite[Thm.~1.17]{FKM2} when $W$ is the
constant $1$. For smooth $W$, the same proof applies, except that we apply partial summation in \cite[(3.2)] {FKM2} if $m_1 \not= m_2$ to remove the weight $W(m_1n/Y) W(m_2n/Y)$; this produces a factor $Q$ that after taking square roots produces the above bound.

\section{Correlation sums of Kloosterman sums and Hecke eigenvalues}

In this section we prove Proposition \ref{propIScuspidal}.  We replace
the tempered Voronoi summation formula by the Voronoi summation
formula for cusp forms, which we state in a form suited to our purpose.

\begin{proposition}[Voronoi summation formula for cusp forms with arithmetic weights modulo
  primes]\label{prvoronoi} 
  Let $q $ be a prime. Let $W$ be a smooth function compactly supported in
  $]0,\infty[$ and let $f$ be a holomorphic cuspidal Hecke eigenform
  of level $1$ and weight $k$.  Let $\eps(f)=\pm 1$ denote the sign of
  the functional equation of the Hecke $L$-function $L(f,s)$ and let
  \[
  \widetilde W(y)=\int_{0}^\infty W(u)\mathcal{J}_k(4\pi\sqrt{ uy})du,
\]
where
\begin{equation*}
  \mathcal{J}_k(u)  =
  2\pi i^kJ_{k-1}(u). 
  \end{equation*}
Then, for any $q$-periodic arithmetic function $K\colon \Zz\to\Cc$, we have
\[
\sum_{n\geq 1}\lf(n)K(n)W\Bigl(\frac nN\Bigr)= \frac{\what K(0)}{q^{1/2}}\sum_{n\geq 1} \lf(n)W\Bigl(\frac nN\Bigr)+\\
\eps(f)\frac{N}{q} \sum_{n\geq 1}\lf(n)\widecheck K( n)\widetilde
W\Bigl(\frac{nN}{q^2}\Bigr).
\]
In particular, for   $a$ coprime to $q$, we have
\begin{multline*}
  \sum_{n\geq 1} \lf(n)\Kl(an;q)W\Bigl(\frac nN\Bigr)=
  \eps(f)\frac{N}{q^{1/2}} \sum_{n\equiv a \mods q
  }\lf(n)\widetilde W\Bigl(\frac{nN}{q^2}\Bigr)
  - \eps(f)\frac{N}{q^{3/2}} \sum_{n\geq 1 }\lf(n)\widetilde
  W\Bigl(\frac{nN}{q^2}\Bigr).
  \end{multline*}
\end{proposition}

\begin{proof}
  We expand $K(n)$ into additive characters
\[
K(n)=\frac{1}{q^{1/2}}\sum_{a\mods q}\what K(a)e_q(-an)
\]
and apply the classical summation formula
\[
\sum_{n\geq 1} \lf(n)W\Bigl(\frac{n}N\Bigr)e\Bigl(-\frac{an}{q}\Bigr) =
\eps(f)\frac{N}{q} \sum_{n\geq
  1}\lf(n)e\Bigl(\frac{\overline{a}n}{q}\Bigr) \widetilde
W\Bigl(\frac{Nn}{q^2}\Bigr),
\]
valid for all $N>0$ and all $a$ coprime to $q$ (\cite[Theorem A.4]{KMVDMJ}).
\end{proof}

We can now easily prove Proposition \ref{propIScuspidal}: integration
by parts shows that for any $A\geq 0$ and $\eps>0$ we have
\[\widetilde W\Bigl(\frac{nN}{q^2}\Bigr)\ll_{k,A,\eps}
q^\eps\Bigl(1+\frac{nN}{q^2}\Bigr)^{-A}\] (see \cite[Lemma 2.4]{445}),
so that (using Deligne's bound  $|\lf(n)|\leq d(n)\ll_\eps n^\eps$), we get
\[
\sum_{n} \lf(n)\Kl(an;q)W\Bigl(\frac nN\Bigr)\ll_{\eps,k} (qN)^\eps
\Bigl(q^{1/2}+\frac{N}{q^{1/2}}\Bigr).
\]

\section{Application to the fourth moment of Dirichlet \texorpdfstring{$L$-functions}{L-functions}}
In this section we prove Theorem \ref{472}. The general strategy of
the proof has been explained in detail in our paper \cite{445}. We
assume some familiarity with this paper, and refer in particular to
\cite[\S 1.2, \S 6.1, \S 6.3]{445} for notations. 

We begin with the unconditional bound. Let
\begin{multline*}
  B_{E, E}^{\pm}(M,N) =\frac{1}{(MN)^{1/2}}\sum_{\substack{m\equiv\pm n\mods q \\ m \not=  n}}{d(m)d(n)}W_1\Bigl(\frac{m}M\Bigr)W_2\Bigl(\frac{n}N\Bigr)\\
  - \frac{1}{q(MN)^{1/2}} \sum_{m, n} d(m) d(n)
  W_1\Bigl(\frac{m}M\Bigr)W_2\Bigl(\frac{n}N\Bigr).
\end{multline*}
Our objective is to prove that
for $\eta=\bexp$ one has
\begin{equation}\label{eqgoal}
B_{E,E}^{\pm}(M,N)-\mathrm{MT}^{od,\pm}_{E,E}(M,N)=:\mathrm{ET}_{E,E}^{\pm}(M,N)\ll_\eps q^{-\eta+o(1)},	
\end{equation}
where $\mathrm{MT}^{od,\pm}_{E,E}(M,N)$ is a suitable main term (described in \cite{MY}) and $M,N$ range over a set of $O(\log^2q)$ real numbers satisfying
\[1\leq M\leq N,\ MN\leq q^{2+o(1)}\]
(the first bound is by symmetry, the second is the length of the approximate functional equation). We set
\[N^*=q^2/N,\ M=q^\mu,\ N=q^\nu,\ \nu^*=2-\nu,\]
so that
\[ 0\leq \mu\leq\nu,\quad  -\varepsilon \leq \nu^*-\mu. \]

In view of the bound \cite[(3.18)]{445}, which reads
$$\ET^{\pm}_{E,E}(M,N)  \ll q^{\eps}\Bigl( \frac{N}{qM}\Bigr)^{1/4} \left(1+ \Bigl( \frac{N}{qM}\Bigr)^{1/4} \right)$$ 
and which  is proved using spectral theory, we may also assume that
\begin{equation}\label{4eta}
\mu+\nu^*\leq 1+4\eta
\end{equation}
for otherwise \eqref{eqgoal} is certainly true. 
Proceeding in the same way as in \cite[\S 6.3]{445}, we apply Voronoi summation to  reduced to
the following bounds for $O(\log^4 q)$ sums of the shape
\begin{multline*}
  S^\pm(M_1,M_2,M_3,M_4)=\frac{1}{(qMN^*)^{1/2}}\sumsum_{m_1,m_2,m_3,m_4}
  W_1\left(\frac{m_1}{M_1}\right)W_2\left(\frac{m_2}{M_2}\right) \\
\times  W_3\left(\frac{m_3}{M_3}\right)W_4\left(\frac{m_4}{M_4}\right)
  \Kl(\pm m_1m_2m_3m_4;q)\ll q^{-\eta+o(1)},
\end{multline*}
where the $W_i$ satisfy~(\ref{Wbound}) with $Q=q^{\eps}$, and the
$M_i$ written in the shape $M_i=q^{\mu_i}$, $i=1,2,3,4$, satisfy
\[ \mu_1\leq \mu_2\leq \mu_3\leq\mu_4,\quad 0\leq
  \mu_1+\mu_2+\mu_3+\mu_4 = \mu+\nu', \quad \nu' \leq \nu^*. \] By the
trivial bound for Kloosterman sums (and recalling \eqref{4eta}), we may assume that
\begin{equation}\label{munu*range}
1-2\eta\leq \mu+\nu'\leq \mu+\nu^*\leq 1+4\eta,
\end{equation}
for otherwise \eqref{eqgoal} is true. 

We use the same strategy as in \cite[\S 6.3]{445}, except that we
replace \cite[(5.5)]{445} by Proposition \ref{propSZ}.  Thus, if the largest
variables $m_3,m_4$ are large enough, we apply \eqref{SZboundsmooth} to
them (fixing $m_1,m_2$); otherwise, we find it more beneficial to
group variables differently producing a bilinear sum of Kloosterman
sums to which we apply Proposition \ref{proptypeII}.

Explicitly, using \eqref{SZboundsmooth} we obtain that
\begin{align*}
  S^\pm(M_1,M_2,M_3,M_4) 
  &\ll q^{o(1)} \frac{M_1M_2}{(qMN^*)^{1/2}}\Bigl(q^{1/2}+\frac{M_3M_4}{q^{1/2}}\Bigr)\\
  & \ll
    q^{o(1)}\Bigl(\sqrt{\frac{M_1M_2}{M_3M_4}}+\frac{(MN')^{1/2}}q\Bigr)
    \ll q^{o(1)}\Bigl(\sqrt{\frac{M_1M_2}{M_3M_4}}+q^{-\eta}\Bigr)
\end{align*}
since $q^{\frac12(1+4\eta)-1}\leq q^{-\eta}$.  We may therefore assume
that
\begin{equation}\label{SZcondition}
0\leq \mu_3+\mu_4-(\mu_1+\mu_2)\leq 2\eta.	
\end{equation}
We now apply Proposition \ref{proptypeII} with ${\tt M}=M_4$ and ${\tt N}=M_1M_2M_3$ so that ${\tt MN}=q^{\mu+\nu'}\leq MN^*$ and derive  
\[ S^\pm(M_1,M_2,M_3,M_4)\ll q^{o(1)}\big(q^{\frac{\mu_1+\mu_2+\mu_3-1}2}+q^{-\frac{1}4+\frac{\mu_4}2}\big). \]
We claim that under the current assumptions both exponents on the right hand side are $\leq - \eta$, which completes the proof. Indeed,  since $\mu_4\geq \mu_i$ for $i=1$,
$2$, $3$,  we obtain by \eqref{munu*range} that 
\[ \Bigl(1+\frac13\Bigr)(\mu_1+\mu_2+\mu_3)\leq \mu_1+\mu_2+\mu_3+\mu_4\leq 1+4\eta\implies 
\mu_1+\mu_2+\mu_3\leq \frac34+3\eta, \]
hence 
\[ {\frac{\mu_1+\mu_2+\mu_3-1}2}\leq {-\frac{1}8+\frac{3}2\eta}\leq {-\eta}. \]
Moreover, by \eqref{SZcondition} and \eqref{munu*range} (since
$\mu_1\leq \mu_2\leq \mu_3\leq \mu_4$) we have
\[ \mu_4\leq 2\eta+\mu_1+\mu_2-\mu_3\leq 2\eta+\mu_1\leq 2\eta+\frac{1}3(1+4\eta-\mu_4)=\frac{1}3+\frac{10}3\eta-\frac13\mu_4,\]
which implies that $\mu_4\leq \frac14+\frac52\eta,$ 
and so
\[ -\frac{1}4+\frac{\mu_4}2\leq-\frac18+\frac54\eta\leq-\eta. \]

If the Ramanujan--Petersson conjecture is available, we can use
\cite[(1.7)]{445} with $\theta = 0$ in place of \cite[(3.2)]{445} and
replace \eqref{4eta} with $\mu+\nu^*\leq 1+2\eta.$ Then the same
strategy leads to the numerical value $\eta = 1/16$.

\section{Sums of Kloosterman sums along the primes: proof of Theorem \ref{Klpq}}  
 \subsection{Proof of inequality \eqref{evening2}} 

 We now recall the main ideas of the proof of \cite[Thm.~1.5]{FKM2},
 since our proof will follow the same path until the moment we use
 Proposition \ref{propSZ}. We will incorporate some shortcuts and
 combinatorial improvements to \cite{FKM2}, mainly due to the
 assumption $X \leq q$.  By \cite[p.~1711--1716]{FKM2}, we are
 reduced to proving the same bound as \eqref{evening2} for the sum
\[
\mathcal S_{W,X} (\Lambda, \Kl):= \sum_n \Lambda (n) \Kl (n;q) W
\Bigl( \frac{n}{X} \Bigr),
\]
where $\Lambda$ is the von Mangoldt function.  We now apply
Heath-Brown's identity \cite{HB} with integer parameter $J\geq 2$. This
decomposes $\mathcal S_{W,X} (\Lambda, \Kl)$ into a linear
combination, with coefficients bounded by $O_J (\log X)$, of
$O(\log^{2J} X)$ sums of the shape
\begin{multline}\label{HBdecomp}
  \Sigma (\uple{M}, \uple{N})=\underset{m_1, \dots, m_J}{\sum\cdots
    \sum}
  \alpha_1(m_1) \alpha_2(m_2) \cdots \alpha_J (m_J)\\
  \times \underset{n_1, \dots, n_J}{\sum \cdots\sum} V_1
  \Bigl(\frac{n_1}{N_1}\Bigr) \cdots V_J \Bigl(\frac{n_J}{N_J}\Bigr) W \Bigl(
  \frac{m_1\cdots m_J n_1\cdots n_J}{X}\Bigr) \Kl (m_1\cdots
  m_Jn_1\cdots n_J;q)
\end{multline}
where
\begin{itemize}
\item $\uple{M}= (M_1, \dots, M_J)$, $\uple{N} =(N_1, \dots, N_J)$
  are $J$-tuples of parameters in $[1/2, 2X]^{2J}$ which satisfy
\begin{equation}\label{restr}
N_1 \geq N_2 \geq \cdots \geq N_J, \ \ M_i \leq X^{1/J}, \ \ M_1\cdots M_J N_1 \cdots N_J\asymp_J X;
\end{equation}
\item the arithmetic functions $m\mapsto \alpha_i (m)$ are bounded and
  supported in $[M_i/2, 2M_i]$;
\item the smooth functions $x\mapsto V_i (x)$ satisfy~(\ref{Wbound})
  with parameter $Q$.
\end{itemize}

It now remains to study    the sum $\Sigma (\uple{M}, \uple{N})$ defined in \eqref{HBdecomp}  for every $(\uple{M}, \uple{N})$ as above. We estimate $\Sigma(\uple{M},\uple{N})$ in two ways.

Our first method is to  bound  $\Sigma (\uple{M}, \uple{N})$ by applying \eqref{SZboundsmooth1} to the largest smooth variables $n_1$ and $n_2$ in 
$\Sigma(\uple{M},\uple{N})$ and a trivial summation over the other variables. We obtain 
\begin{equation*}
\Sigma(\uple{M},\uple{N}) \ll q^\varepsilon Q^2 X \Bigl( \frac{q^{1/2}}{N_1 N_2}+ \frac{1}{q^{1/2}}\Bigr),
\end{equation*}
which, by \eqref{restr} and the assumption $X\leq q$,  simplifies into
\begin{equation}\label{ineq3}
\Sigma(\uple{M},\uple{N}) \ll q^\varepsilon Q^2 X\ \bigl( q^{1/2}/ (N_1N_2)\bigr).
\end{equation}

Our second method is to apply Proposition \ref{proptypeII} to  $\Sigma (\uple{M}, \uple{N})$; in this way we obtain  
\begin{equation}\label{ineq1}
\Sigma (\uple{M}, \uple{N}) \ll q^\varepsilon Q^{1/2}\,X\, \Bigl( \frac{1}{M^{1/2}} +\frac{q^{1/4}}{(X/M)^{1/2}}\Bigr)
\end{equation}
for any factorization 
\[
M_1\cdots M_JN_1\cdots N_J =M\times N.
\]

We have now to play with \eqref{ineq3} and \eqref{ineq1} in an optimal way to bound $\Sigma (\uple{M}, \uple{N})$.  We follow the same presentation as in \cite[\S 4.2]{FKM2}. We introduce the real numbers $\kappa$, $x$, $\mu_i$, $\nu_j$,  $1\leq i,j\leq J$,  defined by
\[ Q=q^\kappa,\  X=q^x,\ M_i=q^{\mu_i},\ N_j=q^{\nu_j} \]
and we set
\[( \uple{m}, \uple{n}) =(\mu_1, \dots, \mu_J, \nu_1, \dots, \nu_J) \in [0, x]^{2J}.
\]

The conditions \eqref{restr} are reinterpreted as
\begin{equation}\label{restr1}
  \sum_i \mu_i +\sum_j \nu_j =x\leq 1, \quad \mu_i \leq x/J, \quad \nu_1 \geq \nu_2\geq \cdots \geq \nu_J.
\end{equation}

According to \eqref{ineq3} and \eqref{ineq1}, we introduce the function (compare with \cite[definition (4.5)]{FKM2})
$\eta (\uple{m}, \uple{n})$ defined by
\begin{equation}\label{defeta}
\eta (\uple{m}, \uple{n}):= \max\Bigl\{ 
 (\nu_1+\nu_2) -\frac{1}{2}-2 \kappa \ ;\,  \max_\sigma \min
\Bigl(  \frac{\sigma}{2}, \frac{x-\sigma}{2} -\frac{1}{4} \Bigr) -\frac{\kappa}{2}
\Bigr\},
\end{equation}
where $\sigma$ ranges over all possible sub-sums of the $\mu_i$ and $\nu_j$ for $1 \leq i, j \leq J$, that is, over the sums
\[
\sigma =\sum_{i \in \mathcal I} \mu_i +\sum_{j \in \mathcal J} \nu_j,
\]
for $\mathcal I $ and $\mathcal J$ ranging over all possible subsets of $\{1, \dots, J\}$.

With these conventions, as a consequence of \eqref{ineq3} and \eqref{ineq1} we have the inequality
\[
\Sigma(\uple{M},\uple{N}) \ll (qQ)^\varepsilon  q^{-\eta (\uple{m}, \uple{n})}\,X,
\]
  and finally, summing aver all possible $(\uple M, \uple N)$, we have the inequality
\begin{equation}\label{793}
\mathcal S_{W,X} (\Lambda, \Kl) \ll (qQ)^\varepsilon \, q^{-\eta}\, X,
\end{equation}
where 
\[
\eta = \min_{(\uple{m}, \uple{n})} \eta (\uple{m}, \uple{n}), 
\]
where $( \uple{m}, \uple{n})$ satisfy \eqref{restr1}. 

The estimate~\eqref{evening2} is trivial for $x<3/4$, so we may assume that $3/4 \leq  x \leq 1$. For $\varepsilon >0$ sufficiently small, let $\mathcal I_x$ be the interval
\[
\mathcal I_x = [x/6-\varepsilon, x/3+\varepsilon],
\]
and choose $J=10$ to apply Heath-Brown's identity.

We now consider two different cases in the combinatorics of $(\uple{m}, \uple{n})$. 
\begin{itemize}
\item If $(\uple{m}, \uple{n})$  contains a subsum $\sigma \in \mathcal I_x$, then, by \eqref{defeta}, we have the inequality
\[
\eta (\uple{m}, \uple{n}) \geq \min \Bigl(\frac{x/6}{2}, \frac{x-x/3}{2}-\frac{1}{4}\Bigr)-\frac{\kappa}{2}-\frac{\varepsilon}{2},
\]
which simplifies into
\begin{equation}\label{firstbound}
\eta (\uple{m}, \uple{n}) \geq \frac{x}{3} -\frac{1}{4} -\frac{\kappa}{2} -\frac{\varepsilon}{2}.
\end{equation}
\item If $(\uple{m}, \uple{n})$  contains no subsum $\sigma \in \mathcal I_x$, then the sum of all the $\mu_i$ and $\nu_j$ which are less than  $x/6-\varepsilon$ is also less than $x/6-\varepsilon$ (this is a consequence of the inequality $2(x/6-\varepsilon) < x/3 +\varepsilon$). In light of~\eqref{restr1}, this includes all $\mu_i$, and so some $\nu_j$ must be greater than $x/3+\varepsilon$. On the other hand, since $3 (x/3 +\varepsilon) > x$, we deduce that at most two $\nu_i$  (more precisely, $\nu_1$ or $\nu_1$ and $\nu_2$) are greater than $x/3 +\varepsilon$. Combining these remarks, we deduce the inequality
\[
\nu_1+\nu_2 \geq x-(x/6 -\varepsilon) = 5x/6 +\varepsilon,
\]
which implies, by  \eqref{defeta},  the inequality
\begin{equation}\label{secondbound}
\eta (\uple{m}, \uple{n}) \geq \frac{5x}{6}- \frac{1}{2} -2 \kappa -\varepsilon.
\end{equation}
\end{itemize}
By  \eqref{793}, \eqref{firstbound} and \eqref{secondbound}, we deduce the inequality
\begin{equation}\label{828}
\mathcal S_{W,X} (\Lambda, \Kl) \ll  (qQ)^\varepsilon \bigl( q^{1/4} Q^{1/2} X^{2/3} + q^{1/2} Q^2 X^{1/6}
\bigr).
\end{equation}
In the above upper bound, the first  term is larger than the second  one if and only if $Q<q^{-1/6} X^{1/3}$,  and in this case, we have $Q^\varepsilon < q^\varepsilon$. However, when $Q\geq  q^{-1/6} X^{1/3}$, it is easy to see that the bound  \eqref{evening2} is trivial since we have
\[
q^{1/4} Q^{1/2} X^{2/3}\geq q^{1/4} (q^{-1/6} X^{1/3})^{1/2} X^{2/3}=   q^{1/6} X^{5/6} \geq  X,
\]
since we suppose $X\leq q.$ In conclusion, we may drop the second term on the right-hand side of \eqref{828}. This remark completes the proof of \eqref{evening2}.

\subsection{Proof of inequality \eqref{evening1}}  The proof mimics the proof appearing in \cite[\S 4.3]{FKM2}. By a simple subdivision, it is sufficient to prove the inequality
\begin{equation}\label{836}
 \sum_\stacksum{X<p\leq\frac32X}{p\text{ prime}}\Kl(p;q)\ll q^{1/6+\eps}\, X^{7/9}.
\end{equation}
Let $\Delta <1/2$ be some parameter, let $W$ be a smooth function defined on $[0, +\infty[$ such that
\[
\supp (W) \subset [1-\Delta, \textstyle \frac{3}{2}+\Delta], \ 0\leq W \leq 1, \ W(x) =1 \text{ for } 1 \leq x \leq \frac{3}{2}, 
\]
and such that the derivatives satisfy
\[
x^j W^{(j)} (x)\ll_j Q^j,
\]
with $Q=\Delta^{-1}$. By applying \eqref{evening2}, we have
\begin{align*}
 \sum_\stacksum{X<p\leq \frac{3}{2}X}{p\text{ prime}}\Kl(p;q)& \ll \Delta X + 1 + \Bigl\vert\, \sum_p W \Bigl( \frac{p}{X}\Bigr) \Kl (p;q)\, \Bigr\vert\\
 & \ll \Delta X +  q^{1/4 +\varepsilon} Q^{1/2} X^{2/3} \ll q^{1/6 +\varepsilon} X^{7/9},
  \end{align*}
by the choice $\Delta =q^{1/6} X^{-2/9} < 1/2$ (the claim is trivial if $q^{1/6} \geq \frac{1}{2}X^{2/9}$). This completes the proof of \eqref{836}.


\begin{bibdiv}

\begin{biblist}


     
\bib{445}{article}{
 author={V. Blomer},
 author={\' E. Fouvry},
 author={E. Kowalski},
 author={Ph. Michel},
 author={D. Mili\'cevi\' c},
 title={On moments of twisted $L$-functions},
 journal={Amer. J. Math.},
 date={to appear, \url{arXiv:1411.4467}},
 }


 


 \bib{FKM2}{article}{
   author={{\'E}. Fouvry },
   author={E. Kowalski},
   author={Ph. Michel},
   title={Algebraic trace functions over the primes},
   journal={Duke Math. J.},
   volume={163},
   date={2014},
   number={9},
   pages={1683--1736},
}
 
\bib{FKM1}{article}{
  author={{\'E}. Fouvry},
  author={E. Kowalski},
  author={Ph. Michel},
  title={Algebraic twists of modular forms and Hecke orbits},
 journal={Geom. Funct. Anal.},
  volume={25},
  date={2015},
  number={2},
  pages={580--657}, 
 }
   
  
\bib{FKMd3}{article}{
  author={{\'E}. Fouvry},
  author={E. Kowalski },
  author={Ph. Michel },
  title={On the exponent of distribution of the ternary divisor function},
  journal={Mathematika},
  volume={61},
  date={2015},
  number={1},
  pages={121--144},
} 




\bib{HB}{article}{
 author = {D. R. Heath-Brown},
 title = {Prime numbers in short intervals and a generalized Vaughan identity},
 journal={Canad. J. Math.},
 volume={34}, 
 date={1982}, 
 pages={1365--1377},}


\bib{KMVDMJ}{article}{
   author={E. Kowalski },
   author={Ph. Michel},
   author={J. VanderKam},
   title={Rankin-Selberg $L$-functions in the level aspect},
   journal={Duke Math. J.},
   volume={114},
   date={2002},
   number={1},
   pages={123--191},
 
}
 

\bib{SZ}{article}{
   author={I. Shparlinski},
   author={T. P. Zhang},
   
   title={Cancellations amongst Kloosterman sums},
   journal={Acta Arith.},
   note={(to appear, \url{arXiv:1601.05123})},
}


 
\bib{MY}{article}{
 author={M. P. Young},
 title={The fourth moment of Dirichlet $L$-functions},
 journal={Ann. of Math. (2)},
 pages={1--50},
date={2011},
volume={173},
number={1},
}


 



\end{biblist}

\end{bibdiv}
\end{document}